\documentclass[12pt]{article}
\usepackage{amssymb,amsmath,amsthm,latexsym}
\title{A family of semifields in characteristic $2$}

\author{
Daniele Bartoli\\
Departments of Mathematics \\
Ghent University\\
B 9000 Gent, Belgium\\
\\
J\"urgen Bierbrauer\\
Department of Mathematical Sciences\\
Michigan Technological University\\
Houghton, Michigan 49931 (USA)\\
\\
Gohar Kyureghyan \\
Institut f\"{u}r Algebra und Geometrie \\
Otto-von-Guericke Universit\"{a}t Magdeburg\\
 D-39106 Magdeburg, Germany \\
\\
Massimo Giulietti, Stefano Marcugini, Fernanda Pambianco\\
Dipartimento di Matematica e Informatica\\ Universit\`a degli Studi di Perugia \\ 06123 Perugia, Italy
\thanks{The research of D. Bartoli, M. Giulietti, S. Marcugini, and F. Pambianco
was supported in part by Ministry for Education, University
and Research of Italy (MIUR) (Project PRIN 2012 ''Geometrie di Galois e
strutture di incidenza'')
 and by the Italian National Group for Algebraic and Geometric Structures
and their Applications (GNSAGA - INdAM).
The work of D. Bartoli was supported also by the European
Community under a Marie-Curie Intra-European Fellowship (FACE
project: number 626511).
J. Bierbrauer's research was supported in part by GNSAGA - INdAM.}}

\begin{document}
\maketitle
\newtheorem{Theorem}{Theorem}
\newtheorem{Proposition}{Proposition}
\newtheorem{Lemma}{Lemma}
\newtheorem{Definition}{Definition}
\newtheorem{Corollary}{Corollary}
\newtheorem{Example}{Example}
\newtheorem{Problem}{Problem}
\newtheorem{Remark}{Remark}
\def\nz{\mathbb{N}}
\def\gz{\mathbb{Z}}
\def\rz{\mathbb{R}}
\def\ef{\mathbb{F}}
\def\CC{\mathbb{C}}
\def\o{\omega}
\def\p{\overline{\omega}}
\def\e{\epsilon}
\def\a{\alpha}
\def\b{\beta}
\def\g{\gamma}
\def\d{\delta}
\def\l{\lambda}
\def\s{\sigma}
\def\t{\tau}
\def\O{\Omega}
\def\sm{\setminus}
\def\bsl{\backslash}
\def\la{\longrightarrow}
\def\arr{\rightarrow}
\def\ov{\overline}
\newcommand{\D}{\displaystyle}
\newcommand{\T}{\textstyle}

\begin{abstract}
We construct and describe the basic properties of a family of semifields in characteristic $2.$
The construction relies on the properties of projective polynomials over finite fields.
We start by associating non-associative products to each such polynomial. The resulting
presemifields form the degenerate case of our family. They are isotopic to the Knuth semifields
which are quadratic over left and right nucleus. The non-degenerate members of our family
display a very different behaviour. Their left and right nucleus agrees with the center, the middle
nucleus is quadratic over the center. None of those semifields is isotopic or Knuth equivalent to a
commutative semifield. As a by-product we obtain the complete taxonomy of the characteristic $2$ semifields
which are quadratic over the middle nucleus, bi-quadratic over the left and right nucleus and not
isotopic to twisted fields. This includes {determining} when
two such semifields are isotopic and the order of the autotopism group.
\end{abstract}


\section{Introduction}
A finite {\bf presemifield} of order $q=p^r$ ($p$ a prime) is an algebra $(F,+,*)$ of order $q$
which satisfies the axioms of the field of order $q$ with the possible exception of the associativity
of multiplication and the existence of an identity element of multiplication. A presemifield is a {\bf semifield}
if in addition an identity element of multiplication exists.
The addition in a presemifield may be identified with the addition in the field of the same order.
A presemifield is {\bf commutative} if its multiplication is commutative.
A geometric motivation to study (pre)semifields comes from the fact that there is a
bijection between presemifields  and projective planes of the same order which are
translation planes and also duals of translation planes.
Presemifields $(F,+,*)$ and $(F,+,\circ )$ of order $q=p^r$ are defined to be {\bf isotopic} if there
exist elements $\a_1, \a_2, \b\in GL(r,p)$ such that $\b (\a_1(x)*\a_2(y))=x\circ y$ always holds.
This equivalence relation is motivated by the geometric link as well. In fact,
two presemifields are isotopic if and only if they determine isomorphic projective planes
(see Albert \cite{Albert60}).
\par
General constructions of semifields which give families of examples that exist in arbitrary characteristic
and in each characteristic $p$ for an infinity of dimensions $r$ are hard to come by.
A classical example are the Albert twisted fields \cite{Alberttwist}.

Recently, a new family of
presemifields in odd characteristic $p$ has been defined in \cite{JBprojgeneral} by using the theory of projective polynomials and Albert twisted fields as ingredients. This is a large family, since it contains the
Budaghyan-Helleseth family of odd characteristic commutative semifields (see \cite{BuHe11}) and
an infinity of semifields which are not isotopic to commutative semifields. Examples of the new family exist for each order $q=p^r$ where $p$ is a prime and $r=2m$ is even.

The aim of the present paper is to construct and investigate an analogous of such a family in characteristic $2$. A first step in this direction was taken in \cite{semichar2ACCT}. Our approach
 is based on the {projection method}, as described in  \cite{JBprojgeneral}. Basic ingredients for our construction are {\bf projective polynomials.}
We use the theory of projective polynomials over finite fields as given in Bluher \cite{Bluher}.
The definition of the characteristic $2$ family $B(2,m,s,l,t)$ is in Subsection \ref{basicdefsub}. The underlying
projective polynomial is
$p_{s,t}(X)=p_1X^{2^s+1}+p_2X^{2^s}+p_3X+p_4\in\ef_{2^m}[X]$ in Definition \ref{Bfamilyintrodef}.
\par
In fact, the theory of presemifields sheds some more light on the theory of projective polynomials.
This is seen in Section \ref{assocprodsection} where we associate a multiplication $*$ on $\ef_{2^{2m}}$
to each such projective polynomial $p_{s,t}(X)\in\ef_{2^m}[X]$ (see Definition \ref{assproddef})
in such a way that $p_{s,t}(X)$ has no zeroes in $\ef_{2^m}$ if and only if the algebra $(F,+,*)$ is a
presemifield (Theorem \ref{multKnuthsecondchar2theorem}). The resulting presemifields form the degenerate case
$l=0$ of our characteristic $2$ family. We use this link to define the generic family $B(2,m,s,l,t)$
where $0\not=l\in\ef_{2^m}$ in Subsection \ref{basicdefsub} and to study its properties later on.
In particular we prove that none of the presemifields $B(2,m,s,l,t)$ is isotopic to a commutative semifield
(Section \ref{noncommutsection}). This contrasts with the odd characteristic case where the Budaghyan-Helleseth
family of commutative semifields is contained in our family and it remains an open problem if our family
contains commutative examples which do not belong to the Budaghyan-Helleseth family.
A similar feature concerns the nuclei. Here $x\in F$ belongs to the {\bf left nucleus} of a semifield $(F,+,*)$
if the associativity equation $x*(y*z)=(x*y)*z$ holds for all $y,z.$ Analogous statements characterize the
{\bf middle nucleus} and the {\bf right nucleus.}
Isotopic semifields have isomorphic nuclei. The nuclei correspond to
certain important subgroups of the collineation group of the corresponding
projective plane. We determine the nuclei of the semifields isotopic to $B(2,m,s,l,t)$ in Section \ref{nucleisection}.
Again this is a more complete result than in the odd characteristic case where the determination of the middle
nucleus remains an open problem. As a by-product of our results in a parametric special case
we obtain a complete characterization of the semifields in characteristic $2$ which are quadratic over one of the nuclei,
quartic over the center and are not isotopic to generalized twisted fields.\\
The smallest order in which examples of our generic family exist is $256.$ The corresponding
presemifields $B(2,4,2,l,t)$ where $0\not=l\in L, l^5\not=1$ and $t=[p_1,p_2,p_3,p_4]\in L^4$ is legitimate
in the sense of Definition \ref{Bfamilyintrodef} come in three isotopy classes, each with autotopism group of
order $450$ (see Section \ref{s=m/2section}). They correspond to three isomorphism classes of
projective semifield planes of order $2^8$
each of which has $450\times 2^{24}$ collineations.
\par
In Subsection \ref{standardp=2sub} we introduce compact notation. The definition of our family is in
Subsection \ref{basicdefsub}.
We close this introduction with a detailed description of the results of this paper in Subsection \ref{structuresub}.

\subsection{A standard situation in characteristic $2$}
\label{standardp=2sub}

All our semifields have even dimension $r=2m.$
Let $F=GF(2^{2m})\supset L=GF(2^m)$ and $T,N:F\la L$ the norm and trace functions. Let $\mu\in L$ be of absolute trace $=1$ and
$z\in F$ such that $z^2+z=\mu .$ Then $z\notin L$ and we use $1,z$ as a basis of $F\vert L.$ In particular we write $x=a+bz=(a,b)$ where $a,b\in L$ and refer to
$a,b$ as the real and imaginary part $Re(x)$ and $Im(x),$ respectively.
\par
Let $0\leq s<2m$ and $x\mapsto x^{\s}$ be the corresponding field automorphism, where $\s =2^s,$ let $K_1=\ef_{2^{\gcd(m,s)}}$ be the fixed field of $\s$ in $L.$
Then $z^4=z^2+\mu^2=z+\mu^2+\mu .$ Continuing like that we obtain the following:

\begin{Lemma}
Let $\mu_s=\sum_{i=0}^{s-1}\mu^{2^i}.$ Then
$z^{\s}=z+\mu_s$ and $x^{\s}=(a^{\s}+\mu_sb^{\s},b^{\s}).$
\end{Lemma}

In particular $\mu_0=0, \mu_1=\mu, \mu_2=\mu +\mu^2$ and $\mu_m=tr_{L\vert\ef_2}(\mu )=1$
(because of the transitivity of the trace), and $\ov{z}=z^{2^m}=z+1.$
Further $\mu_{s+m}=\mu_s+1.$
We have $\ov{x}=(a+b,b), T(x)=Im(x)=b,$ and
$$(a,b)(c,d)=(ac+\mu bd,ad+bc+bd).$$
In particular $1/z=(1/\mu ,1/\mu )$ and
$1/(a,b)=(1/D)(a+b,b),$ where $D=a^2+ab+\mu b^2.$
The conjugates of $x$ are
$$x^2=(a^2+\mu b^2,b^2), x^4=(a^4+(\mu^2+\mu )b^4,b^4),\dots
,x^{\s}=(a^{\s}+(\mu^{2^{s-1}}+\dots +\mu )b^{\s},b^{\s}).$$

\subsection{A family of semifields}
\label{basicdefsub}

\begin{Definition}
\label{Bfamilyintrodef}
Let $m,s, \s$ as in Subsection \ref{standardp=2sub}.
The quadruple $t=[p_1,p_2,p_3,p_4]\in L^4$ is {\bf legitimate} if the polynomial
 $p_{s,t}(X)=p_1X^{\s+1}+p_2X^{\s}+p_3X+p_4$ has no roots in $L.$
Let $\Omega =\Omega (m,s)$ be the set of legitimate quadruples.
Let further $l\in L$ such that either $l=0$ or $l\in L^{*}\sm (L^{*})^{\s-1}.$
The presemifield defined by

\begin{equation}
\label{circgeneralprodequ}
x\circ y=(p_1ac^{\s}+lp_1a^{\s}c+p_2bc^{\s}+lp_2a^{\s}d+p_3ad^{\s}+lp_3b^{\s}c+
p_4bd^{\s}+lp_4b^{\s}d,ad+bc)
\end{equation}

will be denoted $B(2,m,s,l,t).$
\end{Definition}

In order to obtain an expression of $x\circ y$ using constants from the larger field $F$ we use
the following terminology:

\begin{Definition}
\label{FLcorresponddef}
Let $C_1=(v_1,h_1), C_2=(v_2,h_2)\in F.$ The quadruple $t=t(C_1,C_2)=[p_1,p_2,p_3,p_4]\in L^4$ {\bf corresponding to} the pair
$(C_1,C_2)\in F^2$ is defined by
$$p_1=h_1+h_2, p_2=v_1+v_2+h_1+h_2, p_3=v_1+v_2+\mu_sh_1+(\mu_s+1)h_2,$$
$$p_4=\mu_sv_1+(\mu_s+1)v_2+(\mu_s+\mu )h_1+(\mu_s+\mu +1)h_2.$$
\end{Definition}

\begin{Proposition}
\label{circFlevelproductprop}
Let $t=t(C_1,C_2)=[p_1,p_2,p_3,p_4].$ Then
\begin{equation}
\label{circFlevelprodequ}
x\circ y=T((C_1\ov{y}^{\s}+C_2y^{\s})x)+lT((C_2y+\ov{C_1}\ov{y})x^{\s})+T(x\ov{y})z.
\end{equation}
\end{Proposition}
\begin{proof} This is a direct calculation, using
$xy^{\s}=(a,b)(c^{\s}+\mu_sd^{\s},d^{\s})=(ac^{\s}+\mu_sad^{\s}+\mu bd^{\s},
ad^{\s}+bc^{\s}+(\mu_s+1)bd^{\s}),$
$x\ov{y}^{\s}=(a,b)(c^{\s}+(\mu_s+1)d^{\s},d^{\s})=(ac^{\s}+(\mu_s+1)ad^{\s}+\mu bd^{\s},
ad^{\s}+bc^{\s}+\mu_sbd^{\s})$ and analogous expressions.
\end{proof}

\subsection{The structure of the paper}
\label{structuresub}

In the remainder of the paper we study the presemifields $B(2,m,s,l,t)$ and the semifields isotopic
to them. The proof that the $B(2,m,s,l,t)$ are indeed presemifields is in Section \ref{lformsection}.
 The multiplication $x\circ y$ in $B(2,m,s,l,t)$ is given in (\ref{circgeneralprodequ}) (on the level of the
field $L$), as well as in Proposition \ref{circFlevelproductprop} in terms of the larger field $F.$
If the automorphism associated to $\s$ is the identity on $L$ (cases $s=0, s=m$),
 it follows from the general form of $x\circ y$ that
$L$ is in the center of a semifield isotopic to $B(2,m,s,l,t)$
(see \cite{JBprojgeneral}, Proposition 3).
This implies that we are in the field case.
It may therefore be assumed that $s\not=0, s\not=m.$
Isotopies are studied in Section \ref{isotopysection}.
In Section \ref{restrisotopsection} we use this to define a group,
direct product of a cyclic group and a group $GL(2,L),$ which permutes our presemifields
(for given $m,s,l$).
\par
Observe that the condition on $l\in L$ is independent of the
conditions on the quadruple $t.$ The special case $l=0$ is degenerate but interesting as the
corresponding semifields are those which are quadratic over left and right nucleus
(Knuth \cite{KnuthJAlg}, see Section \ref{Knuthsemisection}). In the remainder of the paper
we exclude the degenerate case $l=0$ from the discussion.
Cases $m/\gcd(m,s)$ even and $m/\gcd(m,s)$ odd behave rather differently.
It is shown in Section \ref{Cfamilysection} that in the former case the multiplication simplifies.
We study case $s=m/2$ in Section \ref{s=m/2section}. The semifields which are quadratic over one of
the nuclei and quartic over the center have been classified in Cardinali-Polverino-Trombetti \cite{CPT}.
Using this we show that up to equivalence in the Knuth cube (see \cite{KnuthJAlg})
our semifields in case $s=m/2$ are precisely
those characteristic 2 semifields which have this property and are not isotopic to generalized twisted fields
 \cite{Alberttwist} or to Hughes-Kleinfeld semifields \cite{HuK}. We also obtain a complete taxonomy of
those characteristic 2 semifields in Section \ref{s=m/2section}. We determine when two of them are
isotopic and we determine the autotopism groups (Theorem \ref{Bqwcensustheorem}).
In Section \ref{noncommutsection} it is shown that $B(2,m,s,l,t), s\not=0,s\not=m$
is never isotopic to a commutative semifield.
The nuclei of the semifields isotopic to $B(2,m,s,l,t), l\not=0$ are studied in Section \ref{nucleisection}:
the left and right nucleus agree with the center of order $2^{\gcd(m,s)},$
whereas the middle nucleus is a quadratic extension of the center.
\par
We start in Section \ref{assocprodsection} by associating non-associative
products to projective polynomials. This is done here in characteristic $2$
but it works over any positive characteristic. This leads to case $l=0$
of Definition \ref{Bfamilyintrodef} and to Knuth semifields.

\section{The associated product}
\label{assocprodsection}

\begin{Definition}
\label{assproddef}
Let $C_1=(v_1,h_1), C_2=(v_2,h_2)\in F, t=t(C_1,C_2)=[p_1,p_2,p_3,p_4],$

\begin{equation}
P_{C_1,C_2,s}(X)=C_2X^{\s+1}+\ov{C_1}X^{\s}+C_1X+\ov{C_2}\in F[X]
\end{equation}

and

\begin{equation}
\label{assprodequ}
x*y=T((C_1y^{\s}+C_2\ov{y}^{\s})x)+T(xy)z
\end{equation}
be the multiplication {\bf associated to} the projective polynomial $P_{C_1,C_2,s}(X).$
Consider also the isotope

\begin{equation}
x\circ y=x*\ov{y}=T((C_1\ov{y}^{\s}+C_2y^{\s})x)+T(x\ov{y})z
\end{equation}

\end{Definition}

Comparison with (\ref{circFlevelprodequ}) shows that $x\circ y$ in Definition \ref{assproddef}
is precisely the multiplication in $B(2,m,s,0,t).$

\begin{Lemma}
\label{fromptovlemma}
The inverse of the transformation in Definition \ref{assproddef} is given by
$$v_1=(\mu_s+\mu )p_1+\mu_sp_2+p_3+p_4, h_1=\mu_sp_1+p_2+p_3,$$
$$v_2=(\mu_s+\mu +1)p_1+(\mu_s+1)p_2+p_3+p_4, h_2=(\mu_s+1)p_1+p_2+p_3.$$
\end{Lemma}

\begin{Theorem}
\label{multKnuthchar2theorem}
The following are equivalent:
\begin{itemize}
\item $(F,*)$ is a presemifield.
\item $T(C_1x\ov{x}^{\s}+C_2x^{\s+1})\not=0$ for all $0\not=x\in F.$
\item $P_{C_1,C_2,s}(X)$ has no root of norm $1.$
\end{itemize}
\end{Theorem}
\begin{proof}
Assume $x*y=0$ for $xy\not=0.$ The imaginary part shows $y=e\ov{x}$ for some $e\in L.$
The real part shows $T(C_1x\ov{x}^{\s}+C_2x^{\s+1})=0.$
Write this out, divide by $\ov{x}^{\s+1}.$
\end{proof}

\begin{Corollary}
\label{p1nonzerocor}
If the conditions of the previous theorem are satisfied then $p_1\not=0.$
\end{Corollary}
\begin{proof}
Case $X=1$ shows $T(C_1)+T(C_2)=h_1+h_2=p_1\not=0.$
\end{proof}

\begin{Theorem}
\label{multKnuthsecondchar2theorem}
The statements in Theorem~\ref{multKnuthchar2theorem}
 are also equivalent to $p_{s,t}(X)$ (see Definition \ref{Bfamilyintrodef})
having no root in $L.$
\end{Theorem}
\begin{proof}
Assume $x\circ y=x*\ov{y}=0, xy\not=0.$  Use the special case $l=0$ of (\ref{circgeneralprodequ}).
If $d=0$ the imaginary part shows $b=0,ac\not=0.$ Then $p_1=0,$ contradiction.
Let $d\not=0.$ By homogeneity it can be assumed $d=1$ and therefore $a=bc.$
Divide by $b.$
\end{proof}

\begin{Lemma}
\label{multKnuthchar2extralemma}
If the conditions of Theorem~\ref{multKnuthchar2theorem} are satisfied, then $N(C_1)\not=N(C_2).$
\end{Lemma}
\begin{proof}
Assume $N(C_1)=N(C_2),$ equivalently $C_1=z_0C_2\not=0$ for $N(z_0)=1.$ Let $z$ be defined by
$z^{\s}=z_0.$ Then $P_{C_1,C_2,s}(z)=0.$
\end{proof}

\begin{Definition}
\label{legitimatedef}
Given $m,s$ we call a pair $(C_1,C_2)\in F^2$ {\bf legitimate} if the conditions of Theorem \ref{multKnuthchar2theorem}
are satisfied.
\end{Definition}

Observe that $(C_1,C_2)$ is legitimate if and only if $t(C_1,C_2)$ is legitimate, see Definition \ref{Bfamilyintrodef}.
We will see in Proposition \ref{Knuthl=0prop} that the semifields isotopic to the presemifields in Theorem \ref{multKnuthchar2theorem} are precisely those which are quadratic over right and left nucleus.

\section{The presemifield property}
\label{lformsection}

We consider the multiplication $x\circ y$ in $B(2,m,s,l,t),$ see
(\ref{circgeneralprodequ}) or (\ref{circFlevelprodequ}). Let $x*y=x\circ\ov{y}.$ Clearly

\begin{equation}
\label{multFlformchar2equ}
x*y=T((C_1y^{\s}+C_2\ov{y}^{\s})x)+lT((\ov{C_1}y+C_2\ov{y})x^{\s})
+T(xy)z.
\end{equation}

\begin{Theorem}
\label{lformchar2theorem}
$B(2,m,s,l,t)$ in Definition \ref{Bfamilyintrodef} is indeed a presemifield (of order $2^{2m}$).
\end{Theorem}
\begin{proof}
For $l=0$ this is Theorem \ref{multKnuthchar2theorem}. Let $l\not=0,$
assume $x*y=0,xy\not=0.$
The imaginary part shows $y=e\ov{x}$ for $e\in L.$
The real part factorizes:
$(e^{\s}+le)T(C_1x\ov{x}^{\s}+C_2x^{\s+1})=0.$ The first factor is nonzero by the condition on
$l,$ the non-vanishing of the trace term is the second condition of
Theorem \ref{multKnuthchar2theorem}.
\end{proof}

 Observe that the condition on $l$ in Definition \ref{Bfamilyintrodef} can be met for $l\not=0$ only if
$\gcd(m,s)\not=1.$ In particular $m$ has to be a composite number. The smallest choice is therefore
$m=4$ and the resulting semifields have order $2^8.$

\begin{Corollary}
\label{nofieldcor}
$B(2,m,s,l,t)$ where $l\not=0, s\notin\lbrace 0,m\rbrace$ is not isotopic to a field.
\end{Corollary}
\begin{proof}
The restriction to $L$ is $x*y=(h_1+h_2)(xy^{\s}+lx^{\s}y).$
When $l\not=0$ and $\s$ is not the identity on $L,$ then
this is isotopic to a generalized twisted
field.
\end{proof}

\section{Isotopies}
\label{isotopysection}

In this section we study isotopies among the presemifields $B(2,m,s,l,t).$
The opposite of a presemifield $(F,*)$ is defined by $x\circ y=y*x.$
The proof of the following assertion is straightforward.

\begin{Proposition}
\label{Bp=2oppisoprop}
The opposite of $B(2,m,s,l,t(C_1,C_2)),l\not=0$
is isotopic to $B(2,m,s,1/l,t(\ov{C_1},\ov{C_2})).$
Here $t(\ov{C_1},\ov{C_2})=[p_1,p_2+p_1,p_3+p_1,p_4+p_1+p_2+p_3].$
\end{Proposition}

\begin{Proposition}
\label{lformisotopchar2prop}

\begin{itemize}
\item
$B(2,m,s,l,t), l\not=0$ is isotopic to $B(2,m,s,\l^{\s-1}l,t)$ for arbitrary $\l\in L^{*};$
\item $B(2,m,s,l,t)$ is isotopic to $B(2,m,s,l,\l t)$
for arbitrary $\l\in L^{*}$ ({\bf scalar isotopy});
\item
$B(2,m,s,l,[p_1,p_2,p_3,p_4])$ is isotopic to
$B(2,m,s,l^2,[p^2_1,p^2_2,p^2_3,p^2_4])$ ({\bf Galois isotopy});
 \item
 $B(2,m,s,l,[p_1,p_2,p_3,p_4])$ is isotopic to \\
 $B(2,m,s,l,[ k_1^{\s+1}p_1,k_1^{\s}k_2p_2,k_1k_2^{\s}p_3,k_2^{\s+1}p_4])$
for arbitrary $k_1,k_2\in L^{*}$ ({\bf diagonal isotopy}).
\end{itemize}
\end{Proposition}
\begin{proof}
For the first statement use the substitution $x\mapsto\l x, y\mapsto y.$
Scalar isotopy is obvious.
As for Galois isotopy, apply the inverse of the Frobenius map to $a,b,c,d,$ then apply
the Frobenius map to the real and to the imaginary part.
Diagonal isotopy follows from the substitution $a\mapsto k_1a, b\mapsto k_2b,
c\mapsto k_1c, d\mapsto k_2d.$
\end{proof}

Diagonal isotopy is a special case of linear isotopy, as follows.

\begin{Proposition}[{\bf Linear isotopy}]
\label{linearisoprop}
$B(2,m,s,l,[p_1,p_2,p_3,p_4])$ is isotopic to $B(2,m,s,l,[p'_1,p'_2,p'_3,p'_4])$ where
$$p'_1=\a^{\s+1}p_1+\a^{\s}\g p_2+\a \g^{\s}p_3+\g^{\s+1}p_4$$
$$p'_2=\a^{\s}\b p_1+\a^{\s}\d p_2+\b \g^{\s}p_3+\g^{\s}\d p_4$$
$$p'_3=\a\b^{\s}p_1+\b^{\s}\g p_2+\a \d^{\s}p_3+\g\d^{\s}p_4$$
$$p'_4=\b^{\s+1}p_1+\b^{\s}\d p_2+\b \d^{\s}p_3+\d^{\s+1}p_4$$
and $\a ,\b ,\g ,\d\in L$ such that $\a\d\not=\b\g .$
\end{Proposition}
\begin{proof} This corresponds to the substitution $a'=\a a+\b b, b'=\g a+\d b, c'=\a c+\b d, d'=\g c+\d d,$
where $M=\left(\begin{array}{cc}
\a & \b \\
\g & \d \\
\end{array}\right)\in GL(2,L).$
 \end{proof}

 \begin{Corollary}
\label{secondcor}
$B(2,m,s,l,[p_1,p_2,p_3,p_4])$ is isotopic to
$B(2,m,s,l,[1,0,u,v])$ for suitable $v,$ where
$u\in\lbrace 0,1\rbrace .$
\end{Corollary}
\begin{proof}
As $p_1\not=0$ it follows from scalar isotopy that we may assume $p_1=1.$
Linear isotopy with $\a=1,\g=0,\b=p_2\d$ leads to a quadruple $[1,0,*].$
Assume this quadruple has $p_3\not=0.$ Application of linear isotopy to this
quadruple, with $\a =1,\b =\g =0$ yields the claim.
\end{proof}

The following special case of linear isotopy is interesting in its own right.

 \begin{Theorem}
\label{C10p=2addisotheorem}
$B(2,m,s,l,t(C_1,C_2))$ is isotopic to $B(2,m,s,l,t(\a\ov{\a}^{\s}C_1,\a^{\s+1}C_2))$
for all $0\not=\a\in F.$
\end{Theorem}
\begin{proof}
Use Equation (\ref{circFlevelprodequ}) and the substitutions
$x\mapsto\a x, y\mapsto\a y$ for an arbitrary nonzero $\a\in F.$
\end{proof}

\begin{Proposition}
\label{lformswitchCchar2prop}
 $B(2,m,s,l,t(C_1,C_2))$ is isotopic to $B(2,m,s+m,l,t(C_2,C_1)).$
\end{Proposition}
\begin{proof} This follows from basic properties of the trace.
\end{proof}

Note that by Proposition \ref{lformswitchCchar2prop} we may assume $s\leq m.$
The following proposition shows that we may in fact assume $s\leq m/2.$

 \begin{Proposition}
\label{lformsecondisotopprop}
Let $s<m,\s =2^s,\t =2^{m-s}.$ Then  $B(2,m,s,l,[p_1,p_2,p_3,p_4]), l\not=0$ is isotopic to
$B(2,m,m-s,1/l,[p_1,p_3,p_2,p_4]).$
\end{Proposition}
\begin{proof}
Apply $\t$ to $a,b,c,d,$ then divide the real part by $l,$ apply $\s$ to the imaginary part.
\end{proof}

\section{The restricted isotopy group}
\label{restrisotopsection}

\begin{Definition}
Given $m$ and $s,$
the {\bf restricted isotopy group} is the direct product
$G_1=GL(2,L)\times L^{*}$
where $GL(2,L)$ and $L^{*}$ act on the legitimate pairs $(C_1,C_2)$
and on the legitimate quadruples $[p_1,p_2,p_3,p_4]$ by linear isotopy and
scalar isotopy, respectively.
\end{Definition}

Observe that $\vert G_1\vert =(q-1)(q^2-1)(q^2-q),$ where $q=2^m.$

\begin{Lemma}
\label {Omegasizelemma}
Let $\Omega =\Omega (m,s)$ be the set of legitimate quadruples, see Definition \ref{legitimatedef}. Then
$\vert\Omega \vert =\frac{(q+1)q(q-1)^22^d}{2(2^d+1)}.$
Here $d=\gcd(m,s).$
\end{Lemma}
\begin{proof}
Use a formula from Bluher theory \cite{Bluher}:  let $N_0$ be the number of elements $b\in L$ such that
$X^{\s+1}+bX+b$ has no zeroes in $L.$ Then
$N_0=2^{d-1}(2^m+1)/(2^d+1)$ provided $m/d$ is odd, and
$N_0=2^{d-1}(2^m-1)/(2^d+1)$ if $m/d$ is even. \\
Let also $g=\gcd(2^m-1,2^s+1)$ and observe that $g=1$ if $m/d$ odd whereas $g=2^d+1$ if $m/d$ is even.
Elementary counting shows
$$\vert\Omega\vert =(q-1)\lbrace q(q-1-(q-1)/g)+q(q-1)N_0\rbrace =q(q-1)^2(1-1/g+N_0) .$$
In both cases the same formula results.
\end{proof}

We will use the action of $G_1$ on the set $\Omega$ of legitimate quadruples and the fact that for
each $l$ where either $l=0$ or $l\not=0, l\notin (L^{*})^{\s-1}$ legitimate quadruples in the same
orbit under $G_1$ yield isotopic presemifields $B(2,m,s,l,t).$

\section{The degenerate case $l=0:$ Knuth semifields}
\label{Knuthsemisection}

\begin{Proposition}
\label{Knuthl=0prop}
The semifields isotopic to $B(2,m,s,0,t)$ where $s\notin\lbrace 0,m\rbrace$ are precisely those which are
quadratic over the left and the right nucleus (in characteristic $2$).
\end{Proposition}
\begin{proof} The condition on $s$ says  that $K_1$ ( the fixed field
of the automorphism associated to $\s$ in $L$) is properly contained in $L.$ It can be assumed that
$p_1=1,p_2=0,p_3\in\lbrace 0,1\rbrace .$
Start from (\ref{circgeneralprodequ}) for $l=0,$ and
apply the substitution
$b\mapsto b^{1/\s}, c\mapsto c^{1/\s}, d\mapsto d^{1/\s};$  then take the $\s -$th power of the
imaginary part. This leads to
$(a,b)*(c,d)=(ac+p_4b^{1/\s}d,a^{\s}d+bc)$ when $p_3=0,$ and to
$(a,b)*(c,d)=(ac+ad+p_4b^{1/\s}d,a^{\s}d+bc)$ in case $p_3=1.$
When $p_3=0$ this is the standard form given in Knuth \cite{KnuthJAlg}, Section 7.4, type IV, case $g=0.$
In case $p_3=1$ apply the additional substitution $c\mapsto c+d$ to obtain the standard form
in \cite{KnuthJAlg}, Section 7.4, type IV, case $g=1.$
\end{proof}

In the sequel we will always assume $l\not=0.$
We saw in Corollary \ref{nofieldcor} that the corresponding semifields are not fields.

\section{Case $m/\gcd(m,s)$ even: the C-family}
\label{Cfamilysection}
We refer to the semifields isotopic to the $B(2,m,s,l,[1,0,0,p_4]), l\not=0$
as the C-family of semifields.
Let $g=\gcd(q-1,\s +1)$ denote the number of cosets of $(L^{*})^{\s+1}$ in $L^{*}.$

\begin{Lemma}
\label{Cquadruplelemma}
If $m/\gcd(m,s)$ is odd, then there is no legitimate quadruple $[1,0,0,u].$
If $m/\gcd(m,s)$ is even, then there are $q-1-(q-1)/(2^d+1)=2^d(q-1)/(2^d+1)$ legitimate quadruples
$[1,0,0,u]$ where $d=\gcd(m,s).$
\end{Lemma}
\begin{proof}
The quadruple $[1,0,0,u]$ is legitimate if $u\notin  (L^{*})^{\s+1}.$
If $m/d$ is odd, then $g=1$ and $[1,0,0,u]$ is never legitimate.
If $m/d$ is even, then $g=2^d+1.$
\end{proof}

\begin{Lemma}
\label{Cfamilyp=2numberlemma}
Let $m/\gcd(m,s)$ be even. Then
$[1,0,0,u_1]$ and $[1,0,0,u_2]$ are in the same orbit under $G_1$ if and only if
either $u_2\in u_1( L^{*})^{\s+1}$ or  $u_2\in (1/ u_1)( L^{*})^{\s+1}.$
The stabilizer of $[1,0,0,u]$ under $G_1$ has order $(q-1)(2^d+1).$
\end{Lemma}
\begin{proof}
We have $g=2^d+1$ and $K_1=\ef_{2^d}\subseteq (L^{*})^{\s+1}.$
Let the matrix $M$ map $[1:0:0:u_1]\mapsto [1:0:0:u_2].$ We have three conditions:
$$\a^{\s}\b =u_1\g^{\s}\d , \a\b^{\s}=u_1\g\d^{\s} , \b^{\s +1}+u_1\d^{\s +1}=u_2(\a^{\s +1}+u_1\g^{\s +1}).$$
We have $\b =0$ if and only if $\g =0$ which leads to $u_2$ and $u_1$ in the same coset.
Also $\a =0$ iff $\d =0$ and this leads to $u_2$ in the same coset as $1/u_1.$ Assume all entries of
$M$ are nonzero. By homogeneity it can be assumed that $\a =1.$ The first two equations show
$\b =u_1\g^{\s}\d , \b^{\s}=u_1\g\d^{\s} .$ Comparison shows $c=u_1\g^{\s +1}\in K_1.$ This yields the
contradiction $u_1\in (L^{*})^{\s+1}.$
\end{proof}

\begin{Theorem}
\label{mp=2devencounttheorem}
Given $m,s$ and $l\not=0$ such that $m/\gcd(m,s)$ is even, all members of $B(2,m,s,l,t)$ belong to the C-family.
There are $2^{d-1}$ orbits under $G_1,$ where $d=\gcd(m,s).$
\end{Theorem}
\begin{proof} We know from Lemma \ref{Cfamilyp=2numberlemma} that there are precisely $2^{d-1}$
orbits under $G_1$ which belong to the C-family. The stabilizer always has order
$(q-1)(2^d+1),$ so each orbit has length $(q^2-1)(q^2-q)/(2^d+1).$ As there are $2^{d-1}$ such orbits
this exhausts all of $\Omega .$
\end{proof}

\section{The special case $s=m/2$}
\label{s=m/2section}

Case $s=m/2$ is equivalent with $\s\not=1$ but $\s^2=1$ on $L.$
These are the presemifields $B(2,2s,s,l,t), l\not=0.$ We have $m=2s,$ hence $d=s$ and
$m/d=2.$ In particular it follows from Theorem \ref{mp=2devencounttheorem} that we can
assume up to isotopy $p_1=1, p_2=p_3=0, p_4\notin K=\ef_{2^s}.$
\par
For the remainder of this section we will use the following notation:
$$q=2^s, K=\ef_q\subset L=\ef_{q^2}\subset F=\ef_{q^4}.$$
Let $\t :L\la K$ be the trace function. We keep the notation used in Introduction
with respect to a basis $1,z$ of $F\mid L.$

\begin{Definition}
\label{Bqwdef}
Let $w\in L\sm K$ such that $tr_{K\vert\ef_2}(1/\t (w))=0.$
Define a multiplication on $F$ by
$$(a,b)\star (c,d)=(ac+bd^q+wb^qd,a^qd+bc).$$
Let $B_q(w)=(F,\star ).$
\end{Definition}

The condition on $w$ in this definition can be expressed in an equivalent form.

\begin{Lemma}
\label{tracezeroinverselemma}
Let $a\in K^{*}.$ Then $a$ can be written in the form $a=l+1/l$ for some $l\in K$
if and only if $tr_{K\vert\ef_2}(1/a)=0.$
\end{Lemma}
\begin{proof} If $a$ can be written in the required form, then $1/a=l/(1+l^2)=(u+1)/u^2=1/u+1/u^2$
where $u=l+1.$ This shows $tr_{K\vert\ef_2}(1/a)=0.$ The same argument works also in the
opposite direction.
\end{proof}

\begin{Proposition}
\label{Bqwsemiprop}
$B_q(w)$ in Definition \ref{Bqwdef} is a semifield of order $q^4.$
It has middle nucleus $L,$ left and right nucleus $K$ and is not isotopic to a commutative semifield.
\end{Proposition}
\begin{proof}
Clearly $1=(1,0)$ is the unit of multiplication.
The imaginary part vanishes if and only if both $c=ea^q$ and $d=eb$ hold for some $e\in L.$
The real part is then
$e(a^{q+1}+(e^{q-1}+w)b^{q+1}))=0,$ with $e\not=0.$
Let $t =e^{q-1}\not=0$ and observe that $t^q=1/t.$
If $a=0,$ then $w=t.$ If $a\not=0,$ then divide by $ea^{q+1}.$
In both cases $t+w\in K.$ It follows $\t (w)=t+1/t.$
Let $a=\t (w).$ By our assumptions and Lemma \ref{tracezeroinverselemma}, we have
$tr_{K\vert\ef_2}(1/a)=0;$ equivalently $\t (w )=l+1/l$ for some
$l\in K.$ As $t+1/t=l+1/l,$ we have $t\in K$ and therefore $t=1.$ This leads to the contradiction $w\in K.$
The nuclei are determined by a direct calculation.
Another direct calculation using the Ganley criterion \cite{Ganley72} shows non-commutativity.
\end{proof}

Consider now the presemifields $B(2,2s,s,l,[1,0,0,p_4]),l\not=0.$ By isotopy it may be
assumed $l\in K, l\not=1.$

\begin{Proposition}
\label{Bqwcompareprop}
$B(2,2s,s,l,[1,0,0,p_4]),$ with $ l\in K^{*},$ is isotopic to $B_q(w),$ where
$w=(l^2p_4+p_4^q)/(l\t (p_4)).$
\end{Proposition}
\begin{proof}
We have
$$x\circ y=(ac^q+la^qc+p_4bd^q+lp_4b^qd,ad+bc).$$
Substitute $a\mapsto a^q,$ then apply $u\mapsto\frac{1}{l^2+1}(lu+u^q)$ to the real part.
This yields the product
$(ac+\frac{l\t (p_4)}{l^2+1}bd^q+\frac{p_4^q+l^2p_4}{l^2+1}b^qd,a^qd+bc).$
Finally, apply the substitution $b\mapsto\l b, d\mapsto \l d,$ with $\l\in L$ such that
$\l^{q+1}=\frac{l^2+1}{l\t (p_4)}.$
Observe that $w$ does indeed satisfy the condition of Definition \ref{Bqwdef}. In fact,
$\t (w)=(l^2+1)/l,$ hence $1/\t (w)=l/(l^2+1)=1/u+1/u^2,$ with $u=l+1;$ then
$1/\t (w)$ has absolute trace $0.$
\end{proof}

\begin{Corollary}
The semifields $B(2,2s,s,l,t)$
coincide with the $B_q(w),$ up to isotopy.
\end{Corollary}
\begin{proof}
We saw that $p_1=1, p_2=p_3=0$ can be assumed. Also, the mapping
$p_4\mapsto w$ is given by $w=\frac{p_4^q+l^2p_4}{l\t (p_4)}.$
Let now $w$ be given.
As $tr_{L\vert\ef_2}(1/\t (w))=0,$ we can write
$\t (w)=l+1/l$ for some $l\in L.$ Let $p_4=w+1/l.$ This describes the inverse of the mapping above.
\end{proof}

In the sequel we give a complete census of those semifields.

\begin{Theorem}
\label{Bqwcensustheorem}
$B_q(w)$ is isotopic to $B_q(v)$ if and only if $v=\phi (w)$ for some field automorphism
$\phi$ of $L=\ef_{q^2}.$
The order of the autotopism group of $B_q(w)$ is
$2\iota (q^2-1)^2$ where $\iota$ is the order of the stabilizer of $w$ in the Galois
group of $L$ over $\ef_2.$
\end{Theorem}

In the remainder of this section we prove Theorem \ref{Bqwcensustheorem}.
One direction is obvious: $B_q(w)$ is isotopic to $B_q(\phi (w))$ for each field automorphism
$\phi$ of $L.$

\begin{Lemma}
\label{p4lemma}
Let $x*y$ denote the product in $B_q(w).$
The triple $(\a_1,\a_2,\b )$ defines an isotopic semifield with middle nucleus $L,$ left and right
nucleus $K$ and satisfying
$(a,0)\circ (c,0)=(ac,0), (0,b)\circ (c,0)=(0,bc), (a,0)\circ (0,d)=(0,a^qd)$ if and only if
for some field automorphism $\phi$ of $L$ the following hold:
$$\a_1(a,0)=A*\phi (a)=\phi (a)A, \a_2(c,0)=\phi (c)*B=(\phi (c)B_1,\phi (c)^qB_2).$$
$$\a_1(0,c)=C*\phi (c)=(\phi (c)C_1,\phi (c)C_2), \a_2(0,d)=\phi (d^q)*D=(\phi (d^q)D_1,\phi (d)D_2)$$
where $A=(A_1,A_2), \dots ,D=(D_1,D_2)$ are nonzero constants and
the following compatibility conditions are satisfied:

\begin{equation}
\label{D1D2equ}
A_1^qD_2=B_1C_2, A_2D_1=B_2C_1^q,
\end{equation}

\begin{equation}
\label{wequ}
wA_2^qD_2=B_1C_1+B_2^qC_2,
wB_2C_2^q=A_1D_1+A_2D_2^q.
\end{equation}

\end{Lemma}
\begin{proof}
Let $x\circ y=\b (\a_1(x)*\a_2(y))$  and
$\a_1(1,0)=A, \a_2(1,0)=B, \a_1(0,1)=C, \a_2(0,1)=D.$
Then $x\circ y$ defines a semifield if $\b^{-1}(x)=A*\a_2(x)=\a_1(x)*B$ always holds.
We refer to this as the compatibility condition.
It is easy to check that the middle nucleus $M$ of a semifield $(F,*)$ and the middle nucleus
$M'$ of a semifield $(F,\circ )$ where $x\circ y=\b (\a_1(x)*\a_2(y))$ are related by
$M'=\b (A*M*B)$ where $A=\a_1(1), B=\a_2(1).$ In our case this says
$\b (A*L*B)=L,$ equivalently
$\a_1(a,0)=A*\phi (a)=\phi (a)A, \a_2(c,0)=\phi (c)*B=(\phi (c)B_1,\phi (c)^qB_2).$
Here $\phi :L\la L$ and $\phi (1)=1.$
Compatibility is already satisfied.
\par
Condition $(ac,0)=\b (\a_1(a,0)*\a_2(c,0))$ yields $\phi (ac)=\phi (a)\phi (c);$ in
other words, $\phi$ is a field automorphism of $L.$
\par
Equality $(0,bc)=\b (\a_1(0,b)*\a_2(c,0))$ yields $\a_1(0,bc)=\phi (c)\a_1(0,b).$ The case $b=1$ implies
$\a_1(0,c)=C*\phi (c)=(\phi (c)C_1,\phi (c)C_2).$ Analogously
$(0,a^qd)=\b (\a_1(a,0)*\a_2(0,d))$ yields $\a_2(0,a^qd)=(\phi (a),0)*\a_2(0,b).$
The case $d=1$ implies
$\a_2(0,d)=\phi (d^q)*D=(\phi (d^q)D_1,\phi (d)D_2).$
The compatibility equation is
$$(A_1,A_2)*(\phi (b^q)D_1,\phi (b)D_2)=(\phi (b)C_1,\phi (b)C_2)*(B_1,B_2)$$
for all $b\in L.$ The imaginary part yields Equations (\ref{D1D2equ}),
the real part is
$$\phi (b^q)A_1D_1+g(A_2,\phi (b)D_2)=\phi (b)B_1C_1+g(\phi (b)C_2,B_2),$$
which, by comparing coefficients of $\phi (b)$ and $\phi (b^q),$ yields (\ref{wequ}).
\end{proof}

\begin{Proposition}
Let $x*y$ denote the product in $B_q(w).$
The triple $(\a_1,\a_2,\b )$ defines an isotopy from $B_q(w)$ to $B_q(v)$ if and only if
the conditions of Lemma \ref{p4lemma} are satisfied, as well as

\begin{equation}
\label{onlyphivequ}
\phi (v)A_1B_2^q=(A_2B_1)^q+(C_2D_1)^q, \phi (v)A_2B_1=A_1^qB_2+C_1^qD_2
\end{equation}

and

\begin{equation}
\label{phivandwequ}
\phi (v^q)wA_2^qB_2=A_1B_1+A_2B_2^q+C_1D_1+C_2D_2^q,
\phi (v)(A_1B_1+A_2B_2^q)=w(A_2^qB_2+C_2^qD_2).
\end{equation}

\end{Proposition}
\begin{proof}
The condition is
$$(0,b)\circ (0,d)=(bd^q+vb^qd,0),$$
where $v\notin K.$ Equivalently, for $r=bd^q+vb^qd,$
$$(\phi (b)C_1,\phi (b)C_2)*(\phi (d^q)D_1,\phi (d)D_2)=\a_1(r,0)*(B_1,B_2)=(\phi (r)A_1,\phi (r)A_2)*(B_1,B_2).$$
The imaginary part is
$$\phi (b^qd)C_1^qD_2+\phi (bd^q)C_2D_1=\phi (r^q)A_1^qB_2+\phi (r)A_2B_1.$$
Comparing coefficients this becomes (\ref{onlyphivequ}). In the same manner the real part yields
(\ref{phivandwequ}).
\end{proof}

\begin{Proposition}
If $(\a_1,\a_2,\b )$ defines an isotopy from $B_q(w)$ to $B_q(v)$ where $w,v\in L\sm K$
then at least one of the coefficients $A_1,\dots ,D_2$ must vanish.
$B_q(w)$ is isotopic to $B_q(v)$ if and only if $v=\phi (w)$ for some field automorphism
$\phi$ of $L=\ef_{q^2}.$
\end{Proposition}
\begin{proof}
Note that $1/\t (w)\in K$ has absolute trace $0;$ equivalently, $w+w^q=l+1/l$
for some $l\in K,$ and $\lbrace l,1/l\rbrace$ is uniquely determined by $\t (w).$
As a consequence $w^{q+1}\not=1$ as otherwise we would have $w\in\lbrace l,1/l\rbrace\subset K.$
Assume at first all our parameters $A_1,\dots ,D_2$ are nonzero. By (\ref{D1D2equ}), both $D_1$ and $D_2$ can be eliminated
as
$$D_1=B_2C_1^q/A_2, D_2=B_1C_2/A_1^q.$$
Using this in (\ref{wequ}) we obtain
$(A_1/A_2)(C_1/C_2)^q=w+\frac{A_2B_1^q}{A_1B_2}=w^q+\frac{A_1B_2}{A_2B_1^q}.$
As $w+w^q=l+1/l$ it follows $\frac{A_1B_2}{A_2B_1^q}\in\lbrace l,1/l\rbrace ,$
hence $B_2=l(A_2/A_1)B_1^q$ where $l\in K$ is one of the two values satisfying $l+1/l=\t (w).$
This implies $(C_1/C_2)^q=(A_2/A_1)(l+w^q),$ equivalently $C_1=(l+w)(A_2/A_1)^qC_2.$
\par
Use this in Equations (\ref{onlyphivequ}). The first one simplifies to
$$\phi (v)=(A_1B_1)^{q-1}/l+(l+w)(C_2/A_1)^{q+1},$$
whereas the second becomes
$$\phi (v)=l(A_1B_1)^{q-1}+(l+w^q)(C_2/A_1)^{q+1}.$$
This shows $0=(l+1/l)((A_1B_1)^{q-1}+(C_2/A_1)^{q+1})$ which implies
$(A_1B_1)^{q-1}=(C_2/A_1)^{q+1}=1$ and $\phi (v)=w^q.$
The last equation of (\ref{phivandwequ}) simplifies to $(A_2/A_1)^{q+1}=1/l,$ whereas the first yields
$w^2=l(C_1/A_1)^{q+1}+1.$ This gives $w\in K,$ which is a contradiction.
\par
Next we consider the case $C_2=0.$ Then $A_1D_2=0.$ Assume $A_1=0.$ Then $B_1=0=D_2.$
It follows that if $C_2=0$ then $D_2=0.$ It also follows $A_1=B_1=0.$ Only $A_2,B_2,C_1,D_1$ are in
play and they satisfy $A_2D_1=B_2C_1^q.$ Only Equations (\ref{phivandwequ})
need to be considered. They read as
$$wA_2^qB_2+\phi (v)A_2B_2^q=0,  A_2^qB_2+\phi (v)w^qA_2B_2^q=(C_1D_1)^q.$$
Solving both for $\phi (v)$ and substituting $D_1$ yields the compatibility equation
$A_2^{2q}(1+w^{q+1})=C_1^{q+1}B_2^{q-1};$ equivalently $B_2^{q-1}=A_2^{q-1}$ and
$(C_1/A_2)^{q+1}=1+w^{q+1}.$ The resulting value is
$\phi (v)=w,$ so $v$ is obtained by applying Galois isotopy to $w.$
We count $\iota (q^2-1)^2$ autotopisms (case $v=w$) in this case, where $\iota$ is the order of the
stabilizer of $w$ in the Galois group of $L:$ in fact, we have $\iota$ choices for $\phi ,$ for arbitrary
$A_2,$ then $q+1$ choices for $C_1$ and $q-1$ choices for $B_2.$
\par
Condition $D_2=0$ implies $C_2=0,$ so it can be assumed that $C_2D_2\not=0.$
We have $A_2=0$ if and only if $B_2=0,$ and this implies $C_1=D_1=0.$
Consider this case when $A_1,B_1,C_2,D_2$ are the nonzero parameters. We have
$D_2=B_1C_2/A_1^q.$ Aside of that only the last two equations need to be satisfied.
The penultimate equation is $A_1B_1=C_2D_2^q.$ This yields $A_1^{2q}B_1^{q-1}=C_2^{q+1};$
equivalently $(A_1B_1)^{q-1}=(C_2/A_1)^{q+1}.$ This implies $A_1B_1\in K$ and $C_2/A_1$ in the
non-split torus. The last equation is $\phi (v)=w(C_2/A_1)^{q+1}=w.$
In case $v=w$ we obtain the same number of autotopisms as above.
\par
We can assume now that $A_2B_2C_2D_2\not=0.$
Assume $A_1B_1=0.$ Then $A_1=B_1=0$ and also $C_1=D_1=0.$
Choose $C_2=wA_2^qD_2/B_2^q.$ Then the first equation in (\ref{wequ}) is satisfied.
The second is satisfied if and only if $w^{q+1}=1,$ a contradiction.
Assume $C_1D_1=0$ whereas all remaining constants are nonzero. Then we obtain
$\phi (v)=w(A_2/B_2)^{q-1}$ and $\phi (v)=(1/w^q)(A_2/B_2)^{q-1}.$ This yields
$w^q=1/w,$ which is again a contradiction.
\end{proof}

This completes the proof of Theorem \ref{Bqwcensustheorem}.
\par
As mentioned earlier, the smallest order of interest is 256. We see that there are precisely
3 isotopy types of semifields $B_4(w)$ where the $w's$ are chosen as representatives of the orbits of the
Galois group of $\ef_{16}\vert\ef_2$ on elements $w\in\ef_{16}\sm\ef_4.$ Each of those
semifields has 450 autotopisms.

\subsection*{Comparison}
In this subsection we are going to see that the semifields $B_q(w)$ are Knuth equivalent
to the semifields of order $q^4$ in characteristic $2$ which are quadratic over their middle nucleus,
quartic over their center and are not generalized twisted fields or Hughes-Kleinfeld semifields.
In fact, it has been shown in \cite{CPT} that the semifields of order $q^4$ in characteristic $2$ which are quadratic
over the left nucleus, quartic over the center and are neither
generalized twisted fields nor Hughes-Kleinfeld semifields
are precisely those which can be described up to isotopy by a product

\begin{equation}
\label{CPTuvequ}
x*y=(ac+ubd+wbd^q,ad+bc^q),
\end{equation}

where $u,w\in L^{*}$ satisfy a certain polynomial condition.
We are going to see that this semifield is Knuth equivalent to a semifield $B_q(v).$
Observe at first that the substitution $b\mapsto\l b, d\mapsto\l d$ shows that we can assume $u=1.$
We start from Equation (\ref{CPTuvequ}) with $u=1.$
The opposite is obtained by $x\leftrightarrow y.$ This is the multiplication
$$(ac+bd+wb^qd,a^qd+bc).$$
Next we apply the transpose operation (see \cite{KnuthJAlg}).
In order to do this represent the symplectic form on $F^2=L^4$ by
$$\langle (u_1,u_2,u_3,u_4),(v_1,v_2,v_3,v_4)\rangle =tr(u_1v_3+u_2v_4-u_3v_1-u_4v_2),$$
where $tr:L\la\ef_p$ is the absolute trace on $L.$
The spread space corresponding to the pair $(c,d)$ is
$V_{c,d}=\lbrace (a,b,ac+bd+wb^qd,a^qd+bc)\mid a,b\in L\rbrace .$ When is
$(u,v,U,V)$ in the dual of $V_{c,d}$ with respect to the symplectic form? Using basic properties of the trace
shows that this is equivalent to $U=uc+(vd)^q, V=ud+w^qu^qd^q+vc.$
Choosing $u=a,v=b$ this yields the multiplication in the transpose in the form

\begin{equation}
\label{almostBqwequ}
(ac+(bd)^q,ad+w^q(ad)^q+bc).
\end{equation}

We claim that this is isotopic to $B_q(w).$
Let $f(x)=x+w^qx^q.$ Then $f^{-1}(x)=\kappa f(x)$ where $\kappa =1/(1+w^{q+1})\in K.$
Start from Equation (\ref{almostBqwequ}), let the new imaginary part be the old real part and let
the new real part be the image of the old imaginary part under $f^{-1}.$ This yields the product
$(ad+\kappa (bc+w^q(bc)^q),ac+(bd)^q).$ Apply the substitution $c\mapsto d^q, d\mapsto c$
followed by applying $x\mapsto x^q$ to the imaginary part. This yields the product
$(ac+\kappa (bd^q+w^qb^qd,a^qd+bc).$ Clearly we can choose $\kappa =1$ and obtain
the multiplication in $B_q(w^q)$ which as we know is isotopic to $B_q(w).$
\par
Observe that in this section we obtained a complete taxonomy of the semifields of order $q^4$
in characteristic 2 which are quadratic over the middle nucleus and quartic over left and right nucleus.

\section{Non-commutativity}
\label{noncommutsection}

We use a generalization of the Ganley criterion, Corollary 2 of \cite{JBprojgeneral}. It says that
$(F,\circ )$ is isotopic to a commutative semifield if and only if there is some $v\in F^{*}$ such that
$\a (v\circ x)\circ y$ is invariant under the substitution $x\leftrightarrow y$ for all $x,y.$
Here $\a (x)$ is defined by $\a (x)\circ 1=x$ for all $x.$ Let $v=(v_1,v_2).$

\begin{Theorem}
\label{commlformchar2meventheorem}
 $B(2,m,s,l,t)$ for $s>0$ is not isotopic to a commutative semifield.
\end{Theorem}
\begin{proof}
We may assume $l\not=0.$ Let at first $m/\gcd(m,s)$ be even.
It can be assumed that $\s^2$ is not the identity on $L.$
We have
$Im(\a (v\circ x))=Im(v\circ x)=v_1b+v_2a.$ Considering the imaginary part of the equation we obtain
$$Re(\a (v\circ x))d+(v_1b+v_2a)c=Re(\a (v\circ y))b+(v_1d+v_2c)a.$$
This shows $Re(\a (v\circ y))=v_3c+v_4d$ for some $v_3,v_4\in L$ and $v_3=v_1.$
Now compare the real parts. They yield 8 equations. Four of them simply say $v_3=v_1=0.$
Two of the remaining four equations are redundant. The remaining conditions
are $v_4=lp_4v_2^{\s}$ (the coefficient of $bc^{\s}$) and $p_4v_2=lv_4^{\s}$
(the coefficient of $ad^{\s}$). If $v_4=0$ then $v_2=0,$ which is a contradiction. Assume $v_4\not=0.$ Then $v_4^{\s^2-1}l^{\s+1}=v_4^{\s-1}.$
This implies $l^{\s+1}\in (L^{*})^{\s-1},$ which implies $l^2\in (L^{*})^{\s-1}$ and finally
the contradiction $l\in (L^{*})^{\s-1}$
\par
Let now $m/\gcd(m,s)$ be odd. We may assume $p_1=p_3=1, p_2=0.$
The same procedure as above
shows $\a (v\circ x)=(v_1a+v_4b,v_1b+v_2a).$ Comparison of the real parts yields 8 equations as before.
Five of those are redundant. The remaining ones are the following:
the coefficient of $ac^{\s}$ yields $v_1=lv_1^{\s}+lv_2^{\s},$ $bc^{\s}$ yields $v_4=lp_4v_2^{\s}$
and $b^{\s}c$ yields $lv_4^{\s}+lv_1^{\s}=v_1+p_4v_2.$ Use the second equation to eliminate $v_4,$
and then consider $w=v_1+lv_1^{\s}$ instead of $v_1.$ The remaining equations are
$w=lv_2^{\s}=lv_4^{\s}+p_4v_2, v_4=lp_4v_2^{\s}.$ After division by the leading term this yields
$X^{\s^2}+BX^{\s}+CX=0,$ where $B=1/(lp_4)^{\s}, C=1/(l^{\s+1}p_4^{\s-1}).$
 If our presemifield is isotopic to a commutative semifield, then this equation has a nonzero root $x\in L.$
Let $y=x^{\s-1}.$ Then $y^{\s+1}+By+C=0.$
A standard substitution shows that this is equivalent to $y^{\s+1}+b(y+1)=0,$ where
$b=B^{\s+1}/C^{\s}=1/p_4^{2\s}.$
On the other hand the condition from Definition \ref{Bfamilyintrodef}
says that $X^{\s+1}+X+p_4$ has no root in $L.$ This is equivalent to $X^{\s+1}+(1/p_4)(X+1)$ having no root,
which contradicts the condition that we just obtained.
\end{proof}

\section{The nuclei}
\label{nucleisection}

\begin{Theorem}
\label{p=2leftrightnucleitheorem}
Let $0<s<m,l\not=0.$ Then $K_1$ is the center, the right and the left nucleus of the semifield
associated to $B(2,m,s,l,t), l\not=0.$
\end{Theorem}

\begin{proof}
We know that $K_1$ is in the center.
Then it will be enough to prove that the left and right
nucleus have at most $2^{\gcd(m,s)}$ elements. It can be assumed that $s\leq m/2.$
The case $s=m/2$ has been handled in
Section \ref{s=m/2section}, so we may actually assume $0<s<m/2.$
By Proposition \ref{Bp=2oppisoprop} it is enough to
consider the right nucleus.
We work with the multiplication $x*y=x\circ\ov{y}$ of Equation (\ref{multFlformchar2equ}) which we write in polynomial
form as
$x*y=\sum_{k\in E}c_k(y)x^{2^k}$ where $E=\lbrace 0,m,s,m+s\rbrace$ and

$$c_0(y)=C_1y^{\s}+C_2\ov{y}^{\s}+yz,
c_m(y)=\ov{C_2}y^{\s}+\ov{C_1}\ov{y}^{\s}+\ov{y}z$$
$$c_s(y)=l(\ov{C_1}y+C_2\ov{y}), c_{m+s}(y)=\ov{c_s(y)}.$$
In particular $y$ is recovered from $y=c_0(y)+\ov{c_m(y)}.$
\par
 The right nucleus is in bijection with the invertible linear mappings
$V(x)=\sum_{i=0}^{2m-1}a_ix^{2^i},$ where $a_i\in F$
such that for each $y\in F$ there exists
$u=u(y)\in F$ satisfying $V(x*y)=x*u$ (see \cite{JBprojgeneral}, Theorem 3).
In coordinates this means

\begin{equation}
\label{rightnukeKantorequ}
\sum_{i=0}^{2m-1}a_i(\sum_{k\in E}c_k(y)x^{2^k})^{2^i}=\sum_{k\in E}c_k(u)x^{2^k}
\end{equation}

Let $j\notin E.$ The coefficient of
$x^{2^j}$ in $(\ref{rightnukeKantorequ})$ shows

\begin{equation}
\label{rightnukeKantorsecondequ}
a_jc_0(y)^{2^j}+a_{j+m}\ov{c_m(y)}^{2^j}+(a_{j-s}+a_{j-s+m})c_s(y)^{2^{j-s}}=0.
\end{equation}

This is a polynomial equation in $y.$ The coefficient in $y^{2^{s+j}}$ shows
$(a_j+a_{j+m})C_1^{2^j}=0.$ The coefficient of $y^{2^{s+m+j}}$ shows
$(a_j+a_{j+m})C_2^{2^j}=0.$ As $C_1,C_2$ do not both vanish it follows
$a_{j+m}=a_j.$
Use this and  $c_0(y)+\ov{c_m(y)}=y$ in (\ref{rightnukeKantorsecondequ}):
$$a_jy^{2^j}+(a_{j-s}+a_{j-s+m})c_s(y)^{2^{j-s}}=0.$$
This shows $a_j=a_{j+m}=a_{j-s}+a_{j-s+m}=0.$ In particular $a_j\not=0$ only if $j\in E.$
\par
We have $a_s+a_{s+m}=0,$ hence $V(x)=a_0x+a_m\ov{x}+a_s(x^{\s}+\ov{x}^{\s}).$
Comparing coefficients of $x^{\s}$ and $\ov{x}^{\s}$ in (\ref{rightnukeKantorequ}) shows
$a_0+a_m\in L,a_s\in L.$ It follows $u=c_0(u)+\ov{c_m(u)}=\l y,$ where $\l =a_0+\ov{a_m}\in L.$
It follows $a_s=0.$ The formula for $c_s(u)$ shows that $a_m,a_0\in L.$
The formula for $c_0(u)$ shows $\l =\l^{\s}$ and finally $a_m=0,a_0\in L.$

\end{proof}

\begin{Theorem}
\label{p=2middlenucleustheorem}
The middle nucleus of the semifield associated to $B(2,m,s,l,t), s>0, l\not=0$ is a quadratic
extension of the center.
\end{Theorem}

\begin{proof}
Let at first $m/\gcd(m,s)$ be even.
We know that we can assume $p_1=1, p_2=p_3=0$ in this case and have
$x\circ y=(ac^{\s}+la^{\s}c+p_4bd^{\s}+lp_4b^{\s}d,ad+bc).$
It can be assumed that $\s^2$ is not the identity on $L,$ the case $s=m/2$ having been handled in
Section \ref{s=m/2section}.
The condition to determine the middle nucleus is $V(x)\circ y=x\circ z$ where $z=z(y)$
(see \cite{JBprojgeneral}, Theorem 2).
An obvious polynomial argument shows that $V(x)$ has the
form $V(a,b)=(Aa+Bb,Ca+Db),$ and $z=(Ec+Fd,Gc+Hd).$ Comparison of the imaginary parts shows
$E=D,F=B,G=C,H=A,$ hence $z=(Dc+Bd,Cc+Ad).$ Compare the real parts. The coefficients of
$ac^{\s},bc^{\s},a^{\s}c,b^{\s}c,ad^{\s},bd^{\s},a^{\s}d,b^{\s}d$ yield 8 equations for the unknowns
$A,B,C,D\in L.$ The latter four are redundant, the first four are $A=D^{\s},B=p_4C^{\s}, A^{\s}=D, B^{\s}=p_4C.$
Assume $B\not=0.$ Then $B^{\s^2-1}=p_4^{\s-1},$ hence $p_4/B^{\s+1}\in K_1.$ This shows that
$p_4\in (L^{*})^{\s+1},$ contradicting the existence condition of $p_4.$ It follows $B=C=0$ and
$A\in K_2,$ the fixed field of $\s^2, D=A^{\s}.$ The solutions are $V(a,b)=(Aa,A^{\s}b),$ where $A\in K_2.$
\par
Let now $m/\gcd(m,s)$ be odd.
It can be assumed that $p_1=1, p_2=0,p_3=1.$ Our presemifield product is
$x\circ y=(ac^{\s}+la^{\s}c+ad^{\s}+lb^{\s}c+p_4bd^{\s}+lp_4b^{\s}d,ad+bc).$ As before
we have $V(a,b)=(Aa+Bb,Ca+Db),z=(Dc+Bd,Cc+Ad)$
and the condition is $V(x)\circ y=x\circ z.$
Compare the real parts. As before 8 equations arise, the latter half of which are redundant. The other four are
$$A=D^{\s}+C^{\s}, B=p_4C^{\s}, A^{\s}+C^{\s}=D, B^{\s}+D^{\s}=D+p_4C.$$
Use the first two to eliminate $A,B.$
This shows that the middle nucleus is in bijection with the space of $(C,D)\in L^2$ satisfying

\begin{equation}
\label{lastequ}
C^{\s}+C^{\s^2}=D+D^{\s^2}, \;\; p_4C+p_4^{\s}C^{\s^2}=D+D^{\s}.
\end{equation}

Combining those two equations yields
$C^{\s}+C^{\s^2}+(p_4C)^{\s}+p_4^{\s^2}C^{\s^3}=p_4C+p_4^{\s}C^{\s^2};$ equivalently
$f(C)=p_4^{\s}C^{\s^2}+C^{\s}+p_4C=\l\in K_1.$
Assume that $f(X)$ is not invertible. The substitution $X\mapsto X/p_4$ shows that
$X^{\s+1}+X+p_4^2$ has a root in $L.$
This contradicts the existence condition for
$p_4.$
Let now $\l\in K_1$ be given and $C\in L$ the unique element such that $f(C)=\l .$ The sum of the Equations (\ref{lastequ})
is $\l +C^{\s^2}=D^{\s}+D^{\s^2};$ equivalently, $D+D^{\s}=\l +C^{\s}.$ In order to obtain solutions
$(C,D)$ it must be the case that $Tr (\l +C^{\s})=0$ where $Tr$ is the trace $:L\la K_1.$
As $Tr(\l )=\l$ we must show $Tr(C)=\l .$
Clearly we are done
when this has been proved. We apply the method used in \cite{HelKol}. In fact,
$$Tr(C)^2=Tr(C^{\s +\s})=Tr(C^{\s}(p_4^{\s}C^{\s^2}+p_4C+\l ))=Tr(\l C^{\s})=\l Tr(C).$$
This shows $Tr(C)\in\lbrace 0,\l\rbrace .$ Assume $Tr(C)=0.$ Then $C=u+u^{\s}$ for some $u\in L.$ It follows
$\l =f(C)=f(u)+f(u^{\s}).$ Applying this twice shows $f(u^{\s^2})=f(u).$ As $\s$ has odd order $2n+1$ this
yields $f(u^{\s^{2k}})=f(u)$ and finally the contradiction $f(u)=f(u^{\s^{2k+1}})=f(u)+\l .$
\end{proof}

\section{Conclusion}
\label{conclusionsection}

We start from a relation between projective polynomials over finite fields and Knuth semifields which we use for
the construction of a new family of semifields in characteristic $2.$ Those semifields are never isotopic to
commutative semifields. We determine their nuclei. A parametric special case are the characteristic $2$ semifields
of order $q^4$ with middle nucleus of order $q^2$ and center of order $q$ which are different from the
twisted fields and from the Hughes-Kleinfeld semifields. In this case we obtain a complete taxonomy. This includes
the determination of the group of autotopisms.

\end{document}